\documentclass[10pt]{amsart}
\usepackage{amssymb,amsfonts}
\newcommand{\CC}{\mathbb{C}}
\newcommand{\ZZ}{\mathbb{Z}}
\newcommand{\NN}{\mathbb{N}}
\newcommand{\QQ}{\mathbb{Q}}
\newcommand{\FF}{\mathbb{F}}
\newcommand{\RR}{\mathbb{R}}

\DeclareMathOperator{\Hom}{Hom}
\DeclareMathOperator{\Aut}{Aut}
\DeclareMathOperator{\End}{End}
\DeclareMathOperator{\Gal}{Gal}
\DeclareMathOperator{\ord}{ord}

\DeclareMathOperator{\cl}{cl}
\DeclareMathOperator{\dv}{Div}
\DeclareMathOperator{\pic}{Pic}
\DeclareMathOperator{\Lie}{Lie}
\DeclareMathOperator{\Id}{Id}
\DeclareMathOperator{\Perm}{Perm}

\newcommand{\picc}{\pic^0(C_{f,q})}
\newcommand{\picz}{\pic_B^0(C_{f,q})}
\newcommand{\divz}{\dv_B^0(C_{f,q})}
\newcommand{\R}{\mathfrak{R}}
\renewcommand{\O}{\mathcal{O}}

\newcommand{\dq}{\delta_q}

\newcommand{\rf}{\mathfrak{R}_f}

\newcommand{\ii}{\mathfrak{i}}

\numberwithin{equation}{section}
\newtheorem{lem}{Lemma}[section]
{\theoremstyle{remark} \newtheorem{rem}[lem]{Remark}}
\newtheorem{cor}[lem]{Corollary} 
\newtheorem{prop}[lem]{Proposition} 
\newtheorem{thm}[lem]{Theorem}

\begin{document}
\title[Endomorphism algebra of certain Jacobians]{Endomorphism algebras of Jacobians of certain superelliptic curves}
\author{Jiangwei Xue}
\date{}
\maketitle
{\def\thefootnote{}
\footnotetext{\\[-3ex]Department of Mathematics, Pennsylvania
  State University, \\University Park, PA 16802, USA.\\
  Email: xue\_j@math.psu.edu}}


\section{Introduction}
\noindent Let $K$ be a field of characteristic zero, $\bar{K}$ its
algebraic closure. We use $X,Y,Z$ to denote smooth algebraic
varieties over $\bar{K}$.  If $X$ is an abelian variety over $\bar{K}$, we
write $\End(X)$ for its ring of absolute endomorphisms and
$\End^0(X)$ for its endomorphism algebra $\End(X)\otimes \QQ$. Given an
abelian variety $Y$ over $\bar{K}$, we write $\Hom(X,Y)$ for the group of
all $\bar{K}$-homomorphisms from $X$ to $Y$.

Let $p\in \NN$ be a prime, $q=p^r$, $n=mp^s\geq 5$ where $m,r,s\in
\NN$ and $p \nmid m$. Let $f(x)\in K[x]$ be a polynomial of
degree $n$ without multiple roots. We write $\rf=\{\alpha_i\}_{1\leq i
\leq n}$ for the set of roots of $f(x)$ in $\bar{K}$, and $\Gal(f)$ for
the Galois group $\Gal(K(\rf)/K)$. The Galois group may be viewed as a certain
permutation group of $\rf$, i.e., as a subgroup of the group of permutations
$\Perm(\rf)\cong \mathbf{S}_n$.

Let $\zeta_q\in \bar{K}$ be a primitive $q$-th root of unity, and 
\[ \mathcal{P}_q(x)=\frac{x^q-1}{x-1}=x^{q-1}+\dots+1 \in \ZZ[x].\]
Then $\mathcal{P}_q(x)=\prod_{i=1}^{r} \Phi_{p^i}(x)$, where
$\Phi_{p^i}$ denote the $p^i$-th cyclotomic polynomial. Hence
$\QQ[x]/\mathcal{P}_q(x)\QQ[x]$ is isomorphic to
$\prod_{i=1}^{r}\QQ(\zeta_{p^i})$, a direct product of cyclotomic
fields.
 
Let $C_{f,q}$ be a smooth projective model of the affine curve
$y^q=f(x)$.  We denote by $\dq$ the nontrivial periodic automorphism of
$C_{f,q}$:
\[\dq: C_{f,q}\to C_{f,q}, \quad (x,y)\mapsto (x,\zeta_q y).\]
It follows from Albanese functoriality that $\dq$ induces an
automorphism of the Jacobian $J(C_{f,q})$ of $C_{f,q}$, and by an
abuse of notation we also denote the induced automorphism of
$J(C_{f,q})$ by $\dq$. It will be shown (Lemma \ref{lem:min}) that
$\mathcal{P}_q(x)$ is the minimal polynomial of $\dq$ over $\ZZ$ in
$\End(J(C_{f,q}))$.  This gives rise to an embedding
\[\prod_{i=1}^{r}\QQ(\zeta_{p^i}) \cong \QQ[t]/\mathcal{P}_q(t)\QQ[t]
\cong \QQ[\dq] \subseteq \End^0(J(C_{f,q})),\qquad t\mapsto \dq.\] We
prove that if $\Gal(f)$ is large enough, then the embedding above is
actually an isomorphism.
\begin{thm} \label{thm:t1} Let $K$ be a field of
  characteristic zero, $f(x)\in K[x]$ an irreducible polynomial of 
  degree $n\geq 5$. Suppose $\Gal(f)$ is either the
  full symmetric group $\mathbf{S}_n$ or the alternating group
  $\mathbf{A}_n$. Then $\End^{0}(J(C_{f,q}))=\QQ[\dq]=\prod_{i=1}^r
  \QQ(\zeta_{p^i})$.
\end{thm}
\begin{rem} \label{s1:proof} Notice that the cases $p \nmid n$ or
  $q\mid n$, i.e., $s=0$ or $s\geq r$, respectively, have been covered
  in \cite{zarhin05}. So it remains to show that Theorem \ref{thm:t1}
  holds in the case $0<s<r$. Henceforth we will assume that
  $0<s<r$ throughout the rest of the paper.
\end{rem}
\begin{rem} \label{s1:field} Replacing $K$ by a suitable quadratic
  extension if necessary, we may and will assume that
  $\Gal(f)=\mathbf{A}_n$, which is simple nonabelian since $n\geq
  5$. After that, we also note that $K(\zeta_q)$ is linearly disjoint
  from $K(\rf)$, because $K(\zeta_q)/K$ is abelian while the Galois
  group of $K(\rf)/K$ is simple nonabelian. It follows that the Galois
  group $\Gal(f)$ remains $\mathbf{A}_n$ if the field $K$ is replaced
  with $K(\zeta_q)$. So we further assume that $\zeta_q \in K$ for
  the rest of this paper.
\end{rem}
This paper is organized as follows. In Section 2 we establish some
preliminary results about the curve $C_{f,q}$ and its Jacobian
$J(C_{f,q})$. In Section 3 we introduce the abelian subvariety 
$J^{(f,q)}$ and show inductively that $J(C_{f,q})$ is isogenous to a
product of abelian subvarieties $J^{(f,p^i)}$. In Section 4 we show
that $J^{(f,p^i)}$ is absolutely simple with endomorphism algebra
$\QQ(\zeta_{p^i})$, which will enable us to prove the main Theorem.

\textbf{Acknowledgments:} the author would like to express his
gratitude to Yuri G. Zarhin and W. Dale Brownawell, who read 
the draft versions of this paper and provided valuable comments and
suggestions.

\section{Preliminaries}
We keep all the notation and assumptions of the previous section. In
particular, $\deg(f)=n=mp^s$, where $p\nmid m$ and $0<s<r$. It is
known (cf. \cite[page 3358]{Towse}) that on the curve $C_{f,q}$ there
are exactly $p^s$ \emph{points at infinity}, which we denote by
$\infty_j$ with $1\leq j \leq p^s$. We also write $S_{\infty}$ for the
set of points at infinity on $C_{f,q}$.
\begin{lem}\label{lem:inf}
  {\rm (i)} The curve $C_{f,q}$ has genus
  $g(C_{f,q})=\big((q-1)(n-1)+1-p^s\big)/2$. \\
  {\rm (ii)} The automorphism $\delta_q$ acts transitively on $S_\infty$,
  and the stabilizer of each $\infty_j$ is the cyclic group $\langle
  \delta_q^{p^s} \rangle$ generated by $\delta_q^{p^s}$.
\end{lem}
\begin{proof}
  The formula for $g(C_{f,q})$ is a special case of \cite[equation
  (4)]{koo}(See also \cite[Proposition 1]{Towse}). The assertion about
  the action of $\delta_q$ on $S_\infty$ also follows from an explicit
  description of $S_\infty$ given in \cite{Towse}, which we summarize
  here. Recall that $\R_f=\{\alpha_i\}_{1\leq i \leq n}$ is the set of
  roots of $f(x)$. So over the algebraic closure $\bar{K}$, $C_{f,q}$ is
  given by the equation $y^q=c_n\prod_{i=1}^n(x-\alpha_i)$, where
  $c_n\in K$ is the leading coefficient of $f(x)$.   Let $a,b$ be
  a pair of integers such that $am+bp^{r-s}=1$ and $0<a<p^{r-s}$. Let
  us consider the following birational transformation:
  \begin{equation}
    \label{eq:birational}
 x=u^{-a}v^{-p^{r-s}}, \quad y=u^bv^{-m}.     
  \end{equation}
  Then the defining equation for the curve $C_{f,q}$ over $\bar{K}$
  changes into
\[u^{p^s}=c_n\prod_{i=1}^{n}\left(1-\alpha_iu^av^{p^{r-s}}\right).\]
It is easy to show that in $(u,v)$-coordinates 
\[ S_\infty=\{ (\zeta_{p^s}^j, 0)\mid 1\leq j \leq p^s \}\] where
$\zeta_{p^s}=\zeta_q^{p^{r-s}}$ is a primitive $p^s$-th root of
unity. We set $ \infty_j=(\zeta_{p^s}^j,0)$.  To study the action of
$\delta_q$ on $S_\infty$, we first find that the inverse to the
birational transformation (\ref{eq:birational}) is
\[ u=x^{-m}y^{p^{r-s}}, \quad v=x^{-b}y^{-a}.\]
So in $(u,v)$-coordinates, $\delta_q$ is given by 
\[ \delta_q: C_{f,q}\to C_{f,q}, \quad (u,v)\mapsto
(\zeta_{q}^{p^{r-s}}u, \zeta_q^{-a}v)=(\zeta_{p^s}u,\zeta_q^{-a}v).\]
Now assertion (ii) follows easily. 
\end{proof}

Given a smooth algebraic variety $X$ over $\bar{K}$, we write
$\Omega^1(X)$ for the $\bar{K}$-vector space of the differentials of the
first kind on $X$. By functoriality, $\dq$ induces a periodic
$\bar{K}$-linear automorphism $ \dq^*: \Omega^1(C_{f,q}) \rightarrow
\Omega^1(C_{f,q})$. Clearly, $\dq^*$ is semi-simple therefore
$\Omega^1(C_{f,q})$ splits into a direct sum of eigenspaces
$\Omega^1(C_{f,q})_i$, where
\[\Omega^1(C_{f,q})_i:=\{ \omega\in \Omega^1(C_{f,q})\mid
\dq^*(\omega)=\zeta_q^{-i}\omega \}\] is the eigenspace corresponding
to eigenvalue $\zeta_q^{-i}$.  We use $n_i$ to denote the dimension of
the eigenspace $\Omega^1(C_{f,q})_i$. 

If $z\in \RR$ is a real number, then we write $[z]_S$ for the greatest
integer that is {\em strictly} less than $z$. More explicitly, if
$z\not\in \ZZ$, then $[z]_S$ coincides with the floor function $[z]$;
if $z\in \ZZ$, then $[z]_S=z-1$.
\begin{prop}\label{prop:diff}
{\rm (i)} The set 
\[ \left\{ \omega_{i,j}=x^{j-1}dx/y^{q-i} \mid 0<i, 0<j, ni+qj<nq\right\}\] is a
basis of $\Omega^1(C_{f,q})$. \\
{\rm(ii)} $n_i=[ni/q]_S=[mi/p^{r-s}]_S$.

\end{prop}
\begin{proof}
  Part (i) of the proposition is a restatement of \cite[Proposition
  2]{Towse} where we note that on page 3360 of \cite{Towse}, the
  conditions on $(i,j)$ for the differential form $x^idx/y^j$ to be
  holomorphic hold without the assumption $n=\deg f > q$. For part
  (ii), each $\omega_{i,j}=x^{j-1} dx/y^{q-i} $ is an eigenvector of
  $\dq^*$ with eigenvalue $\zeta_q^{-i}$. For a fixed $0<i<q$, the set of
  differential forms $\omega_{q-i,j}\in \Omega^1(C_{f,q})$ forms a
  basis of $\Omega^1(C_{f,q})_i$. The number of $j$'s that satisfy
  $j>0$ and $n(q-i)+qj<nq$ is exactly $[ni/q]_S$. 
\end{proof}

\begin{rem}\label{rem:gcd}
  Let $d(n,q)$ be the greatest common divisor of the integers in the
  set
\[\left\{n_i=\left[\frac{mi}{p^{r-s}}\right]_S=\left[\frac{mi}{p^{r-s}}\right] \Big\vert\: 0<i<q, p\nmid i \right\}.\]
If we write $m=kp^{r-s}+c$, with $0<c<p^{r-s}$, then
  \begin{align*}
    n_1&=
    \left[\frac{kp^{r-s}+c}{p^{r-s}}\right]=k, \qquad\\
    n_{p^{r-s}+1}&=
    \left[\frac{(kp^{r-s}+c)(p^{r-s}+1)}{p^{r-s}}\right]=kp^{r-s}+k+c,
    \qquad \\
    n_{p^{r-s}-1}&=
    \left[\frac{(kp^{r-s}+c)(p^{r-s}-1)}{p^{r-s}}\right]=kp^{r-s}-k+c-1.
    \qquad 
  \end{align*}
It follows that $d(n,q)=1$ when $0<s<r$.
\end{rem}

\begin{rem}
  Similar to the situation with $C_{f,q}$, the automorphism $\dq$ of
  $J(C_{f,q})$ induces a $\bar{K}$-linear automorphism of
  $\Omega^1(J(C_{f,q}))$ which we again denote by $\dq^*$. If $D$ is a
  divisor of degree $0$ on $C_{f,q}$, we write $\cl(D)$ for the linear
  equivalence class of $D$. Let $P_i$ be the $\dq$-invariant point on
  $C_{f,q}(\bar{K})$ with $(x,y)$-coordinates $(\alpha_i,0)$ where
  $\alpha_i\in \R_f$.  For a fixed $i$, the map
\[ \tau: C_{f,q} \to J(C_{f,q}), \qquad P \mapsto \cl((P)-(P_i)) \] is
an embedding of algebraic varieties over $K$, and the induced map
\[\tau^*: \Omega^1(J(C_{f,q}))\to \Omega^1(C_{f,q})\] is an
isomorphism which commutes with $\dq^*$. This allows us to identify
$\Omega^1(J(C_{f,q}))$ with $\Omega^1(C_{f,q})$ via $\tau^*$. In
particular, the dimension of the eigenspace corresponding to the
eigenvalue $\zeta_q^{-i}$ for the automorphism
$\dq^*:\Omega^1(J(C_{f,q}))\to \Omega^1(J(C_{f,q}))$ is exactly
$n_i=[ni/q]_S$ as given by part (ii) of Proposition \ref{prop:diff}. 
\end{rem}

\begin{rem} \label{rm:midq} Clearly, $\zeta_q^{-i}$ is an eigenvalue
  of $\dq^*$ if and only if $n_i>0$. By part (ii) of Proposition
  \ref{prop:diff}, $\zeta_q^{-i}$ is an eigenvalue of $\dq^*$ for all
  $p^r-p^{r-1}\leq i \leq p^r-1\:$(recall that $n=mp^s \geq 5$). Also
  note that $1$ is not an eigenvalue. Taking into account that
  $\dq^q=1$, one sees that
  \[ \mathcal{P}_q(x)=\frac{x^q-1}{x-1}=x^{q-1}+\cdots+x+1\] is the
  minimal polynomial over $\QQ$ of $\dq^*$ on $\Omega^1(J(C_{f,q}))$.
\end{rem}

\begin{lem}\label{lem:min}
  The minimal polynomial over $\QQ$ of $\dq$ in $\End^0(J(C_{f,q}))$
  is $\mathcal{P}_q(t)$. We have natural isomorphisms 
  \begin{gather*}
\prod_{i=1}^{r}\QQ(\zeta_{p^i})\cong \QQ[t]/\mathcal{P}_q(t)\QQ[t]\cong   \QQ[\dq]
  \subseteq \End^0(J(C_{f,q})),\qquad t\mapsto \dq,\\
\ZZ[t]/\mathcal{P}_q(t)\ZZ[t]\cong\ZZ[\dq]
\subseteq \End(J(C_{f,q})),\qquad t \mapsto \dq.
  \end{gather*}
\end{lem}
\begin{proof}
  When $p\nmid \deg f$, the result was obtained in \cite[Lemma
  4.8]{zarhin05}. In our case here, instead of having one \emph{point
    at infinity} on curve $C_{f,q}$ that is fixed by $\dq$ as
  \cite[Lemma 4.8]{zarhin05} dealt with, we have $p^s$ \emph{points at
    infinity} on which $\dq$ acts transitively (Lemma
  \ref{lem:inf}). But it turns out that this cause no extra hindrance,
  and the proof of \cite[Lemma 4.8]{zarhin05} can be carried out here
  in exactly the same way.
\end{proof}

We denote by $B$ the set of all $\dq$-invariant points on the affine
curve $y^q=f(x)$. More explicitly, we set $ B:=\{P_i=(\alpha_i, 0) \mid \alpha_i \in \R_f\}.$

\begin{lem}\label{lem:l1}
  Let $D=\sum_{P \in B}a_P(P)$ be a divisor on $C_{f,q}$ of degree $0$
  that is supported on $B$. Then $D$ is principal if and only if $a_P
  \equiv a_Q \mod q$ for any two points $P,Q\in B$.
\end{lem}
\begin{proof}
  Again, this is a generalized version of \cite[Lemma 4.7]{zarhin05},
  and a similar proof applies.
\end{proof}

\section{Cyclic covers and Jacobians}
Recall that we assumed that $K$ contains the $q$-th roots of unity. Given
a curve $C_{f,q}$ and its Jacobian $J(C_{f,q})$ as in previous sections,
we consider the abelian subvariety
\[ J^{(f,q)}:=\mathcal{P}_{q/p}(\dq)(J(C_{f,q})) \subset J(C_{f,q}).\]
Clearly, $J^{(f,q)}$ is a $\dq$-invariant abelian subvariety which is
defined over $K$. In addition, $\Phi_q(\dq)(J^{(f,q)})=0$, where
$\Phi_q(x)$ denotes the $q$-th cyclotomic polynomial.  Hence we have
the following embeddings
\[\ii: \O=\ZZ[\zeta_q] \hookrightarrow \End(J^{(f,q)}) \qquad E=\QQ(\zeta_q) \hookrightarrow \End^0(J^{(f,q)}) \qquad
\zeta_q \mapsto \dq \vert_{J^{(f,q)}}.\]

\begin{rem}\label{rem:31}
  If $q=p$, then $\mathcal{P}_{q/p}(x)=\mathcal{P}_1(x)=1$ and
  therefore $J^{(f,p)}=J(C_{f,q})$. 
\end{rem}

Let $\End_K(J^{(f,q)})$ be the ring of all $K$-endomorphisms of
$J^{(f,q)}$. By Remark \ref{s1:field},
$\ii(\O)\subseteq \End_K(J^{(f,q)})$.  Let $\lambda=(1-\zeta_q)\O$,
and note that $\lambda$ is the only prime ideal in $\O$ that divides
$p\O$.  We write $J^{(f,q)}_\lambda$ for the group of $\lambda$-torsions
of the abelian variety $J^{(f,q)}$:
\[ J^{(f,q)}_\lambda:=\{ x \in J^{(f,q)} \mid \ii(c)x=0, \forall c \in
\lambda\}.\] Given $z\in J^{(f,q)}$, $\ii(1-\zeta_q)(z)=0$ if and only
if $z$ is fixed by $\dq\vert_{J^{(f,q)}}$.  It follows that $
J^{(f,q)}_\lambda=\big(J^{(f,q)}\big)^{\dq}$, the subgroup of
$J^{(f,q)}$ fixed by $\dq$. Clearly $J^{(f,q)}_\lambda$ is a vector
space over the finite field $k(\lambda):=\O/\lambda\cong\FF_p$. It is
also a $\Gal(\bar{K}/K)$ sub-module of $J^{(f,q)}[p]$, where we denote by
$A[p]$ the subgroup of the abelian group $A$ generated by
elements of order $p$. We would like to
understand the structure of $J^{(f,q)}_\lambda$ as a Galois module.

Recall that $\rf=\{\alpha_i\}_{i=1}^n$ is the set of roots of $f(x)$.  
Let us define
\begin{gather*}
  \FF_p^{\rf}:=\{\phi: \rf \rightarrow
\FF_p\},\\
(\FF_p^{\rf})^0:=\{\phi\in
\FF_p^{\rf}\mid \sum_{i=1}^{n}\phi(\alpha_i)=0\}.
\end{gather*}
The space $\FF_p^{\rf}$ is naturally equipped with a structure of
$\Gal(f)$-module in which both $(\FF_p^{\rf})^0$ and the space of
constant functions $\FF_p\cdot \mathbf{1}$ are submodules. Since
$p\mid n$, the space $\FF_p\cdot \mathbf{1}$ is a $\Gal(f)$-submodule
of $(\FF_p^{\rf})^0$. We write $V_{\rf}$ for the quotient
$\Gal(f)$-module $\FF_p^{\rf}/(\FF_p\cdot \mathbf{1})$, and
$(\FF_P^{\rf})^{00}$ for the image of $(\FF_p^{\rf})^0$ under the
projection map $\FF_p^{\rf}\to V_{\rf}$. Clearly $(\FF_p^{\rf})^{00}$ is
a codimension one submodule of $V_{\rf}$.

Let $\divz$ be the group of divisors of degree $0$ of $C_{f,q}$
supported on the set of points $B=\{(\alpha_i,0)\}_{i=1}^n$. Clearly
$\divz$ is a $\Gal(\bar{K}/K)$-submodule of $\dv^0(C_{f,q})$ for which the
Galois action factors through $\Gal(f)$.  As an abelian group, $\divz$
is free of rank $n-1$. Let $\picz$ be the image of $\divz$ under the
canonical map of $\Gal(\bar{K}/K)$-modules $\dv^0(C_{f,q})\to \picc$. Then
$\picz$ inherits the structure of a $\Gal(f)$-module from $\divz$.  As
before, we write $\picz[p]$ for the subgroup of $\picz$ generated by
elements of order $p$. It is readily seen to be a $\Gal(f)$-submodule
of $\picz$.

\begin{lem} \label{lem:picz2} The $\Gal(f)$-module $\picz[p]$ is
  canonically isomorphic to $V_{\rf}$. So $\dim_{\FF_p}
  \picz[p]=n-1$. Moreover, $\picz[p]$ is a $\Gal(\bar{K}/K)$-submodule of
  $J^{(f,q)}_\lambda$.
\end{lem}
\begin{proof}
  By Lemma \ref{lem:inf}, the automorphism $\dq$ acts transitively on
  the set $S_\infty$ of \emph{points at infinity} of curve $C_{f,q}$,
  and $\dq^{p^s}$ fixes $\infty_j$ for all $\infty_j\in S_\infty$. Further,
  the set $B$ consists all $\delta_q$ invariant points on
  $C_{f,q}$. Pick a point $\infty\in S_\infty$ and
  consider the following divisor classes for $1\leq i \leq n$, 
\[D_i=P_{q/p}((P_i)-(\infty))=
p^{r-1}(P_i)-p^{r-s-1}\sum_{j=1}^{p^s} (\infty_j).\]
Note that 
$pD_i=p^{r}(P_i)-p^{r-s}\sum_{j=1}^{p^s} (\infty_j)=
\dv(x-\alpha_i)=0.$

Recall that in the proof of Lemma \ref{lem:inf} we find a pair of
integers $(a,b)$ such that $a\,m+bp^{r-s}=1$.  Then
\[\begin{split}
  \dv(y^a(x-\alpha_i)^b)&= a\sum_{k=1}^n
(P_k)+bp^r(P_i)-(a\,m+bp^{r-s})\sum_{j=1}^{p^s}(\infty_j)\\
&=a\sum_{k=1}^n(P_k)+bp^r(P_i)-\sum_{j=1}^{p^s}(\infty_j).
\end{split}\]

It follows that the divisor $\sum_{j=1}^{p^s}(\infty_j)$ is linearly equivalent to
$a\sum_{k=1}^n(P_k)+bp^r(P_i)$. Therefore, the divisor class
\[\begin{split}
D_i&=p^{r-1}(P_i)-p^{r-s-1}\sum_{j=1}^{p^s} (\infty_j)\\  
 &= p^{r-1}(P_i)-p^{r-s-1}\left(a(\sum_{k=1}^n(P_k))+bp^r(P_i)\right)\\
 &=-ap^{r-s-1}\sum_{k=1}^n(P_k)+p^{r-1}(1-bp^{r-s})(P_i)\\
 &=a\,mp^{r-1}(P_i)-ap^{r-s-1}\sum_{k=1}^n(P_k).
\end{split}
\]
In particular, $D_i$ is supported on $B$, so $D_i \in \picz[p]$. Let
$\phi_i$ be the function on $\rf$ given by
$\phi_i(\alpha_j)=\delta_{ij}$, where $\delta_{ij}=1$ if $i=j$, and
$0$ otherwise.  The $\phi_i$'s form a basis of $\FF_p^{\rf}$. This
allows us to define an $\FF_p$-linear map $\pi: \FF_p^{\rf}\to
\picz[p]$ by specifying the image of the basis elements,
\begin{equation}
  \label{eq:1}
  \pi(\phi_i)=D_i=P_{q/p}((P_i)-(\infty))=a\,mp^{r-1}(P_i)-ap^{r-s-1}\sum_{k=1}^n(P_k). \end{equation}
 It is clear from the right hand side
of (\ref{eq:1}) that $\pi$ is also a $\Gal(f)$-equivariant map. 

By Lemma \ref{lem:l1}, given an element
$\sum_{i=1}^nx_i\phi_i \in \FF_p^{\rf}$,  
\[
\begin{split}
  \pi(\sum_{i=1}^n x_i\phi_i)=0 &\Leftrightarrow \sum_{i=1}^n
  x_i\left(a\,mp^{r-1}(P_i)-ap^{r-s-1}\sum_{k=1}^n(P_k)\right ) \sim 0 \\
   &\Leftrightarrow \forall(i,j),\: x_ia\,mp^{r-1}\equiv
   x_ja\,mp^{r-1}  \mod q \\
     &\Leftrightarrow \forall(i,j),x_i=x_j \qquad \text{since } p
     \nmid am\\
   &\Leftrightarrow \sum_{i=1}^{n}x_i \phi_i \in \FF_p \cdot \mathbf{1}
\end{split}
\]
Therefore, $\pi$ induces an $\FF_p$-linear embedding $\bar{\pi}:
V_{\rf}=\FF_p^{\rf}/(\FF_p\cdot \mathbf{1}) \hookrightarrow \picz[p]$
of $\Gal(f)$-modules. On the other hand, recall that $\picz$ is
defined to be the homomorphic image of $\divz$, which is free abelian
of rank $n-1$. It follows that $\picz$ can be generated by $n-1$
elements. Hence $\dim_{\FF_p} \picz[p] \leq n-1$. Note that $\dim_{\FF_p}
V_{\rf} =n-1$. We conclude that $\bar{\pi}$ is an isomorphism. It is
clear from our construction that $\picz[p] \subseteq
J^{(f,q)}_{\lambda}$.
\end{proof}

The case when $q=p$ of the previous lemma was treated by B. Poonen and E.
Schaefer (see \cite{PoSch} and \cite{Sch}). 

\begin{rem}\label{rem:tate}
  Let $T_p(J^{(f,q)})$ be the $p$-adic Tate
  module (cf. \cite[p.170]{mumford},\cite[1.2]{serre}) of $J^{(f,q)}$
  defined as the projective limit of Galois modules
  $J^{(f,q)}[p^j]$. Let $\O_\lambda=\O\otimes\ZZ_p$ be the completion
  of $\O$ with respect to the $\lambda$-adic topology. By
  \cite[prop 2.2.1]{ribet}, $T_p(J^{(f,q)})$ is a free $\O_\lambda$
  module of rank $2\dim J^{(f,q)} / [E:\QQ]=2\dim J^{(f,q)}
  /(p^r-p^{r-1})$.
\end{rem}

\begin{prop}\label{prop:dim}
  The abelian variety $J^{(f,q)}$ has dimension $\dim\left(
    J^{(f,q)}\right)= (p^r-p^{r-1})(n-1)/2$.  There is a $K$-isogeny
  $J(C_{f,q}) \to J(C_{f,q/p}) \times J^{(f,q)}$, and the
  $\Gal(\bar{K}/K)$-module $J^{(f,q)}_\lambda$ coincides with
  $\picz[p]$. In particular, the $\Gal(\bar{K}/K)$-action on
  $J^{(f,q)}_\lambda$ factors through $\Gal(f)$.
\end{prop}

\begin{proof}
  Once again the $p\nmid \deg f$ version of the Proposition has
  already been done in
  \cite[Lemma 4.11]{zarhin05}. The proof goes exactly the same way in
  our case. The idea is to do an argument on the dimensions based on
  Lemma \ref{lem:min} and Lemma \ref{lem:picz2}. See the proof of
  \cite[Lemma 4.11]{zarhin05} for more details.
\end{proof}

\begin{cor}\label{cor:iso}
  There is a $K$-isogeny 
\[J(C_{f,q}) \to J(C_{f,p})\times \prod_{i=2}^r
J^{(f,p^i)}=\prod_{i=1}^{r} J^{(f,p^i)}.\]
\end{cor}
\begin{proof}
  This follows from induction on $i$ and the fact that
  $J^{(f,p)}=J(C_{f,p})$, as pointed out in  Remark \ref{rem:31}. 
\end{proof}
\begin{rem} \label{rem:s3las}
  Notice that Theorem \ref{thm:t1} will follow if we show that
  $\End^0(J^{(f,p^i)})\cong \QQ(\zeta_{p^i})$ for all $1\leq i\leq
  r$. The case that $i\leq s$ has already been treated in \cite[Theorem
  4.17]{zarhin05}.
\end{rem}

Let $V$ be a vector space over a field $\FF$, let $G$ be a group and
$\rho: G\to \Aut_\FF(V)$ a linear representation of $G$ in $V$. We
write $\End_G(V)$ for the $\FF$-algebra of $G$-equivariant endomorphisms of
$V$. The equality $J^{(f,q)}_\lambda=\picz[p]$, together with Lemma
\ref{lem:picz2}, enables us to determine
$\End_{\Gal(f)}(J^{(f,q)}_\lambda)$.  



\begin{lem}\label{lem:s2las}
  Let $n=mp^s\geq 5$. Assume that $0<s <r$ and $\Gal(f)$ is either
  $\mathbf{S}_n$ or $\mathbf{A}_n$. Then
  $\End_{\Gal(f)}(J^{(f,q)}_\lambda)=\FF_p\cdot \Id$.
\end{lem}

\begin{proof}
  By Proposition \ref{prop:dim} and Lemma \ref{lem:picz2}(i), we have
  $J^{(f,q)}_\lambda=\picz[p] \cong V_{\rf}$ as $\Gal(f)$-modules. It
  is enough to prove $\End_{\Gal(f)}(V_{\rf})=\FF_p\cdot \Id$.  Recall
  that when $p \mid n$, $(\FF_p^{\rf})^{00}$ is a codimension one
  submodule of $V_{\rf}$.  For simplicity, we write $V$ for $V_{\rf}$,
  and $W$ for $(\FF_p^{\rf})^{00}$.  It is known (Mortimer \cite{Mor}) that $W$
  is a simple $\Gal(f)$-module with $\End_{\Gal(f)}(W)=\FF_p\cdot
  \Id_W$. Given $\theta \in \End_{\Gal(f)}(V_{\rf})$.  the map
  $\theta\vert_{W}: W\to\theta(W)$ is either zero or an isomorphism of
  $\Gal(f)$-modules. In the latter case, unless $\theta(W)= W$, the
  intersection $\theta(W)\cap W$ would be a proper nonzero submodule
  of $W$, contradicting the simplicity of $W$.  Hence
  $\theta\vert_W \in \End_{\Gal(f)}(W)=\FF_p\cdot \Id_W$.  By subtracting an
  element of $\FF_p\cdot \Id_V$, we may and will assume that
  $\theta\mid_W=0$. In order to show that $\End_{\Gal(f)}(V)=\FF_p\cdot
  \Id_V$, it is enough to show that $\theta(V)=0$.  Recall that the
  $\FF_p$-vector space $V$ is generated the set
  $\{\bar{\phi}_i\}_{i=1}^n$, where $\bar{\phi}_i$ denotes the
  equivalent class of $\phi_i$ modulo $\FF_p\cdot \sum_{i=1}^n
  \phi_i$.  We identify $\Gal(f)$ with the permutation group on $n$
  letters $\{1, \ldots, n\}$ and adopt the convention that $\forall g
  \in \Gal(f),\bar{\phi}_i \in V, g(\bar{\phi}_i)=\bar{\phi}_{g(i)}$.
By definition of
  $W=(\FF_p^{\rf})^{00}$, given $g \in \Gal(f)$ and $\bar{\phi} \in
  V$, the element $g\bar{\phi}-\bar{\phi}\in W$. Thus
  $\theta(g\bar{\phi})=\theta(\bar{\phi})$.   
Assume that
\[ \theta(g\bar{\phi})=\theta(\bar{\phi})=\sum_{i=1}^n a_i
\bar{\phi}_i, \qquad \text{and write}\qquad   h=g^{-1}.  \]
Then \[ g\theta(\bar{\phi})= g \sum_{i=1}^n a_i
\bar{\phi}_i=\sum_{i=1}^n a_i
\bar{\phi}_{g(i)}=\sum_{i=1}^n a_{h(i)}
\bar{\phi}_{i}.\]
It follows that \[
\begin{split}
  \theta(g\bar{\phi})=g\theta(\bar{\phi}) &\Leftrightarrow \sum_{i=1}^n a_i
\bar{\phi}_i =\sum_{i=1}^n a_{h(i)}
\bar{\phi}_{i}\\
   &\Leftrightarrow  \sum_{i=1}^n a_i
\phi_i =\sum_{i=1}^n a_{h(i)}
\phi_{i} \mod \FF_p \cdot \sum_{i=1}^n \phi_i\\
&\Leftrightarrow a_i -a_j = a_{h(i)}-a_{h(j)} \qquad \forall 1 \leq i
, j \leq n.
\end{split}
\]
Since $\Gal(f)$ is either $\mathbf{S}_n$ or $\mathbf{A}_n$, it is
doubly transitive. Fix an index $i$. For any pair $(j,k)$ such that
neither $j$ nor $k$ equals to $i$, there exist $g_{jk}\in \Gal(f)$
such that $h_{jk}(i)=g_{jk}^{-1}(i)=i$ and
$h_{jk}(j)=g_{jk}^{-1}(j)=k$. It follows from $\theta g_{jk} = g_{jk}
\theta$ that $ a_i -a_j=a_i-a_k$. Thus $a_j=a_k$ for all pairs $(j,k)$
where neither entry is equal to $i$. Varying the index $i$ shows that
there exist an $a\in \FF_p$ such that $a_j=a$ for all $1 \leq j \leq
n$. Hence $\theta(\bar{\phi})=\sum_{i=1}^n a_i \bar{\phi}_i=a
\sum_{i=1}^n \bar{\phi}_i=0$.
\end{proof}

\section{The endomorphism algebra of $J^{(f,q)}$}
Let $E$ be a number field that is normal over $\QQ$ with ring of
integers $\O$. Let $K$ be a field of characteristic zero that contains
a subfield isomorphic to $E$. Let $(X,\ii)$ be a pair such that $X$ is
an abelian variety over $K$ with an embedding
$\ii(\O)\subseteq \End_K(X)$. The tangent space $\Lie_K(X)$ to $X$ at
the identity carries a natural structure of an $E\otimes_\QQ
K$-module. For each field embedding $\tau: E\hookrightarrow K$, we put
\begin{gather*}
\Lie_K(X)_\tau =\{z\in \Lie_K(X) \mid \ii(e)z=\tau(e)z, \forall e \in
E\}\\
n_\tau(X, E)=\dim_K\Lie_K(X)_\tau   
\end{gather*}

Put $X=J^{(f,q)}$ and $E=\QQ(\zeta_q)$. Let us consider the induced operator
$(\dq \vert_{J^{(f,q)}})^*: \Omega^1(J^{(f,q)})\to
\Omega^1(J^{(f,q)})$.  Applying \cite{zarhin05} Theorem 3.10 to
$Y=J(C_{f,q})$, $Z=J^{(f,q)}$, and $P(t)=\mathcal{P}_{q/p}(t)$, we see
that the spectrum of $(\dq \mid_{J^{(f,q)}})^*$ consists of primitive
$q$-th roots of unity $\zeta^{-i}$ with $n_i>0$, and the multiplicity
of $\zeta^{-i}$ equals $n_i\:$ (see part (ii) of Proposition
\ref{prop:diff}). 
Let $\tau_i: E=\QQ(\zeta_q)\hookrightarrow K$ be the embedding that
sends $\zeta_q$ to $\zeta_q^{-i}$; then $n_{\tau_i}(J^{(f,q)},
\QQ(\zeta_q))=n_i$.

Let $\mathfrak{C}_X$ denote the center of the endomorphism algebra
$\End^0(X)$. We quote the following theorem (cf \cite[Theorem
2.3]{zarhin04}).

\begin{thm}\label{thm:center}
  If $E/\QQ$ is Galois, $\ii(E)$ contains $\mathfrak{C}_X$ and
  $\mathfrak{C}_X \neq \ii(E)$, then there exists a nontrivial automorphism
  $\kappa: E\to E$ such that $n_\tau(X, E)=n_{\tau\kappa}(X, E)$ for
  all embeddings $\tau: E \hookrightarrow K$. 
\end{thm}

For simplicity, let $[a,b]_\ZZ$ denote the set of integers in the
closed interval $[a,b]\subseteq \RR$ that are not divisible by $p$, where $p$
should be apparent from the context. We note that $[1,p^r]_{\ZZ}$ is
the usual set of representatives of the group $(\ZZ/p^r\ZZ)^\times$,
which is isomorphic to $\Gal(\QQ(\zeta_{p^r})/\QQ)$.

The following
two lemmas, together with Theorem \ref{thm:center}, enable us to prove
that $\ii(\QQ(\zeta_q))$ coincides with $\mathfrak{C}_{J^{(f,q)}}$
whenever $\ii(\QQ(\zeta_q))$ contains $\mathfrak{C}_{J^{(f,q)}}$.
\begin{lem} \label{lem:num1} Let $p$ be an odd prime and $q=p^r$ be a
  power of $p$. Given $k \in (\ZZ/q\ZZ)^{\times}$, let $\theta_k$
  denote the action of $k$ on $(\ZZ/q\ZZ)^{\times}$ by multiplication,
  i.e., $\theta_k(u)=ku$ for any $u \in (\ZZ/q\ZZ)^{\times}$.  Let $f:
  (\ZZ/q\ZZ)^{\times} \rightarrow \RR$ be a monotonic function on
  $[1,q]_\ZZ$.  The following statements are equivalent:\\ 
  {\rm{(i)}}$f\circ \theta_k =f$.\\
  {\rm{(ii)}} $k=1$ or $f$ is a constant function.
\end{lem}
\begin{proof}
  Assume that $k \neq 1$. It suffices to show that if $f \circ
  \theta_k =f$, then $f$ is a constant function. Let $\ord(k)$ be the
  order of the element $k$ in $(\ZZ/q\ZZ)^{\times}$. Note that $f
  \circ \theta_k =f$ is equivalent to $f \circ \theta_{k^j} =f$ for
  all $0\leq j < \ord k$.  Since $p$ is an odd prime,
  $(\ZZ/q\ZZ)^{\times}$ is a cyclic group. Given $k$ and $k'$ in
  $(\ZZ/q\ZZ)^{\times}$ with $\ord(k) \mid \ord(k')$, the cyclic group
  $\langle k \rangle$ is a subgroup of $\langle k' \rangle$. Suppose
  the lemma holds for $k$; a fortiori it holds true if one replace $k$
  by $k'$. Thus one may assume that $\ord k$ is a prime. Note if $\ord
  k=2$, then $k=q-1$. Hence $f(q-1)=f(1)$, and $f$ is constant by
  monotonicity. So we further assume that $\ord k$ is an odd prime.

  Any element in $\langle k \rangle$ other than the identity is a
  generator of $\langle k \rangle$. Without loss of generality, we
  assume that $k$ is the smallest in the set
  \[S=\{i \mid 1<i < q, \exists j \text{ such that } i \equiv k^j
  \pmod q\}.\] In other words, $S$ is the set of representatives
  strictly between $1$ and $q$ for non-identity elements of the cyclic
  group $\langle k \rangle$. Also let $m_S$ denote the largest element
  of the set $S$. Notice that $[q/k]<m_S$ by choice of $m_S$, and
  $k<m_S$, since we have more than two elements in $S$.

  From the equality $f(kx)=f(x)$, we conclude $f(m_S)=f(1)$, and
  $f(x)$ is constant on $[1,m_S]_{\ZZ}$ by monotonicity. Notice that
  $f(q-1)=f(q-m_S)$ since $q-m_S \equiv (q-1)m_S \pmod q$. If $q-m_S
  \leq m_S$, i.e., $m_S> (q-1)/2$, then $f(q-m_S)=f(1)$, thus
  $f(q-1)=f(q-m_S)=f(1)$. By monotonicity again, we conclude that
  $f(x)$ is constant.

  So furthermore, assume that $2m_S \leq q-1$. In particular, $[q/k]
  \geq [q/m_S] \geq 2$.  Let $c=[q/k]$ if $p \nmid [q/k]$, and
  $c=[q/k]-1$ otherwise. Clearly, $q-2k< ck < q$. By construction, $c
  \in [1, m_S]_{\ZZ}$, which implies that $f(c)=f(1)$. Therefore,
  $f(ck)=f(c)=f(1)$ and $f(x)$ is constant on $[1,ck]_{\ZZ}$. On
  the other hand,
  \[ ck> q-2k> q-2m_S.\] Hence $f(q-2)=f(q-2m_S)=f(1)$. This shows that
  $f$ is constant on $[1, q-2]_{\ZZ}$. Last, $f(q-1)=f(q-k)$ but clearly,
  $q-k \in [1, q-2]_{\ZZ}$.
\end{proof}

\begin{lem}\label{lem:num2}
Let $p=2$ and $q=2^r$.  Given $k\neq 1$ in $(\ZZ/q\ZZ)^{\times}$,
let $\theta_k$ be as defined in the previous lemma.  Let  $f: (\ZZ/2^r\ZZ)^{\times}
  \rightarrow \RR$ be a monotonic function on $[1,2^r]_\ZZ$.
  If $f\circ \theta_k =f$, then $f$ is constant on 
  $[1,2^{r-1}]_\ZZ$. 
\end{lem}
\begin{proof}
  For $r=1, 2$, the lemma is trivial since $[1,2^{r-1}]_\ZZ$ consists of only one
  element. For $r \geq 3$, $(\ZZ/2^r\ZZ)^{\times} \cong \ZZ/2\ZZ\times
  \ZZ/2^{r-2}\ZZ$. The order of any nontrivial element in
  $(\ZZ/2^r\ZZ)^{\times}$ is a power of $2$. So the cyclic group
  $\langle k \rangle$ contains one of the $3$ elements of order two:
  $2^{r-1}-1, 2^{r-1}+1, 2^{r}-1$. Call it $x_k$. It follows from the
  identity $f\circ \theta_k =f$ that $f(1)=f(x_k)$. The lemma now
  follows from monotonicity of $f$ since $x_k\geq 2^{r-1}-1$.
\end{proof}

\begin{cor}\label{cor:center}
  If $E=\QQ(\zeta_q)$ contains $\mathfrak{C}_{J^{(f,q)}}$, then $
  \mathfrak{C}_{J^{(f,q)}}=\QQ(\zeta_q)$. 
\end{cor}
\begin{proof}
  If $q=2$, then $E=\QQ(\zeta_2)=\QQ$. Since
  $\mathfrak{C}_{J^{(f,q)}}$ is a subfield of $E$, we see that
  $\mathfrak{C}_{J^{(f,q)}}=\QQ=E$. Further assume that $q
  >2$. Suppose that $\mathfrak{C}_{J^{(f,q)}} \neq E$. It follows from
  Theorem \ref{thm:center} that there exists an nontrivial field
  automorphism $\kappa: \QQ(\zeta_q) \to \QQ(\zeta_q)$ such that
  $n_\tau(X,E)=n_{\tau\kappa}(X, E)$ for all embeddings $\tau: E
  \hookrightarrow K$. Clearly $\kappa(\zeta_q)=\zeta_q^k$ for some $k$
  in $[1,q]_\ZZ$. It follows that if we define $f:
  (\ZZ/q\ZZ)^\times \to \RR$ where $ i\mapsto n_i$, then $f(ki)=f(i)$
  for all $i \in (\ZZ/q\ZZ)^\times$. Notice that $f$ is a
  nondecreasing function on the set of representatives. Recall that $n
  \geq 5$. If $p$ is odd, by Lemma \ref{lem:num1}, $f$ is constant.
  If $q=4$, then $k=3$. Thus $n_1=n_3$. If $q=2^r >4$, by Lemma
  \ref{lem:num2}, $f(1)=f(2^{r-1}-1)$. Either way, it contradicts
  part (ii) of Lemma \ref{prop:diff}.
\end{proof}

\begin{prop}\label{prop:p1}
  Let $K$ be a field of characteristic zero. Let $p$ be prime, $q=p^r$
  be a power of $p$, and $n=mp^s\geq 5$ with $0<s<r$ and $p\nmid
  m$. Suppose that $\Gal(f)$ is either $\mathbf{S}_n$ or
  $\mathbf{A}_n$. Then
  \[\End^0(J^{(f,q)})\cong \QQ(\zeta_q),
  \qquad \End(J^{(f,q)})\cong\ZZ[\zeta_q]\]
\end{prop}
\begin{proof}
  By remark \ref{s1:field}, we may assume that $\zeta_q\in K$ and
  $\Gal(f)=\mathbf{A}_n$. Since $n\geq 5$, the group $\Gal(f)$
  contains no non-trivial normal subgroups.  We write
  $\End^0(J^{(f,q)},\ii)$ for the centralizer of $\ii(Q(\zeta_q))$
  inside $\End^0(J^{(f,q)})$. Applying \cite[Theorem
  3.12(ii)(2)]{zarhin07},  and combining with remark \ref{rem:gcd} and Lemma
  \ref{lem:s2las}, we get
  \begin{gather*}
\End^0(J^{(f,q)},\ii)=\ii(\QQ(\zeta_q))\cong \QQ(\zeta_q),    \\
\End(J^{(f,q)},\ii)=\ii(\O)\cong \O.
  \end{gather*} In particular, $\QQ(\zeta_q)$
  contains the center of $\End^0(J^{(f,q)})$. If follows from
  Corollary \ref{cor:center} that $\ii(\QQ(\zeta_q))$ coincides with
  the center of $\End^0(J^{(f,q)})$. Thus
  \[\End^0(J^{(f,q)})=\End^0(J^{(f,q)},\ii)=\ii(\QQ(\zeta_q))\]
  If $\dq\vert_{J^{(f,q)}}$ is the restriction of $\dq$ to $J^{(f,q)}$,
  viewed as an automorphism of $J^{(f,q)}$, then
  $\ii(\O)\cong\ZZ[\dq\vert_{J^{(f,q)}}]$ is the maximal order in
  $\QQ(\zeta_q)$ and $\ZZ[\dq\vert_{J^{(f,q)}}]
  \subseteq \End(J^{(f,q)})$, and we conclude that
  $\ZZ[\dq\vert_{J^{(f,q)}}]=\End(J^{(f,q)})$.
\end{proof}

  \begin{proof}[Proof of Theorem \ref{thm:t1}]
    We need to consider only the case $s>0$ (see Remark
    \ref{s1:proof}).  By Corollary \ref{cor:iso}, it suffices to show
    $\End^0(J^{(f,p^i)}) \cong \QQ(\zeta_{p^i})$ for all $1\leq i\leq
    r$, which follows from Proposition \ref{prop:p1} and Remark
    \ref{rem:s3las}.
  \end{proof}


\begin{thebibliography}{99}
\bibitem{koo} J. K. Koo, \emph{On holomorphic differentials of some algebraic function field of one variable over $\CC$.} Bull. Austral. Math. Soc. 43 (1991), no.3, 399-405.
\bibitem{mumford} D. Mumford, \emph{Abelian varieties}, second edition. Oxford
  University Press, 1974. 
\bibitem{Mor} B. Mortimer, \emph{The modular permutation
    representations of the known doubly transitive groups},
  Proc. London Math. Soc. (3) {\bf 41} (1980), no1, p.1-20. 
\bibitem{PoSch} B. Poonen, E. Schaefer, \emph{Explicit descent for
  jacobians of cyclic covers of the projective line.} J.Reine
  Angew. Math.{\bf 488}, 141-188(1997). 
\bibitem{ribet} K. A. Ribet, \emph{Galois action on division
    points of abelian varieties with real
    multiplications}. Amer. J. Math. {\bf 98}, 751--804(1976).
\bibitem{Sch} E.F. Schaefer, \emph{Computing a Selmer group of a
    Jacobian using functions on the curve}, Math. Ann. {\bf310}
    (1998), no.3, p.447-471.
\bibitem{serre} J.P. Serre, \emph{Abelian $\ell$-adic representations
    and elliptic curves}, Benjamin, New York, 1968. 
\bibitem{Towse} C. Towse, \emph{Weierstrass points on cyclic covers of the projective line.} Trans. Amer. Math. Soc. 348 (1996), no.8, 3355-3378.
\bibitem{zarhin04} Yu.G. Zarhin, \emph{Endomorphism rings of Jacobians
    of cyclic covers of the projective line}. Math. Proc. Cambridge
  Philos. Soc. {\bf 136} (2004), no.2, p.257-267.
\bibitem{zarhin05} Yu.G. Zarhin, \emph{Endomorphism algebras of
    superelliptic jacobians}.  In: Bogomolov, F., Tschinkel, Yu.(eds.)
  Geometric Methods in Algebra and Number Theory. Progress in
  Mathematics, vol.235, pp.339-362. Birkh\"auser, Boston (2005).
\bibitem{zarhin07} Yu.G. Zarhin, \emph{Endomorphisms of superelliptic
    jacobians}. Math. Z., {\bf 261}(2009), 691-707, 709.
\end{thebibliography}
\end{document}